\documentclass[11pt]{article}

\usepackage{amsmath}
\usepackage{amssymb}
\usepackage{amsthm}
\usepackage{bm}
\usepackage{color}
\usepackage{graphicx}
\usepackage{mathtools}
\usepackage{hyperref}
\usepackage{enumerate}
\usepackage[shortlabels]{enumitem}
\usepackage[colorinlistoftodos,prependcaption,textsize=footnotesize]{todonotes}

\usepackage{hyperref}
\hypersetup{
    colorlinks=true,     % false: boxed links; true: colored links
    linkcolor=blue,      % color of internal links
    citecolor=blue,      % color of links to bibliography
    filecolor=blue,      % color of file links
    urlcolor=blue        % color of external links
}

\newtheorem{definition}{Definition}[section]
\newtheorem{lemma}[definition]{Lemma}
\newtheorem{proposition}[definition]{Proposition}

\newtheorem{example}[definition]{Example}

\newtheorem{corollary}[definition]{Corollary}

\newtheorem{theorem}[definition]{Theorem}
\newtheorem*{proposition*}{Proposition}
\newtheorem*{opQ}{Question}
\theoremstyle{remark}

\newcommand{\interior}{\mathrm{int}\,}
\newcommand{\reInt}{\mathrm{ri}\,}
\newcommand{\SOC}[2]{{\mathcal{L}^{#2} _{#1}}}
\newcommand{\norm}[1]{\lVert{#1}\rVert}
 
\newcommand{\inProd}[2]{\langle #1 , #2 \rangle }

\newcommand{\cS}{\mathcal{S}}

\newcommand{\psdcone}[1]{\cS_+^{#1}}
\newcommand{\pdcone}[1]{\cS_{++}^{#1}}

\newcommand{\RR}{\mathbb{R}}

\newcommand{\rank}{\operatorname{rank}}

\newcommand{\Aut}{\operatorname{Aut}}
\newcommand{\tr}{\textup{tr}}

\newcommand{\stdCone}{ {\mathcal{K}}}

\newcommand{\ambSpace}{ \mathcal{V} }

\renewcommand{\Re}{\mathbb{R}}   

\title{Minimal hyperbolic polynomials and ranks of homogeneous cones}

\author{
	Jo\~ao Gouveia\thanks{CMUC, Department of Mathematics, University of Coimbra, 3001-454 Coimbra, Portugal.
		During part of this work, the author also held a visiting professorship at the Institute of Statistical Mathematics, Japan.
		This author was supported by the Centre for Mathematics of the University of Coimbra (UIDB/00324/2020,
		funded by the Portuguese Government through FCT/MCTES).
		Email: \href{jgouveia@mat.uc.pt }{jgouveia@mat.uc.pt}.} \and
	Masaru Ito		\thanks{Department of Mathematics, College of Science and Technology, Nihon University,
		1-8-14 Kanda-Surugadai, Chiyoda-Ku, Tokyo 101-8308, Japan. This author was supported partly by the JSPS Grant-in-Aid for Early-Career Scientists  21K17711.
		Email: \href{ito.masaru@nihon-u.ac.jp}{ito.masaru@nihon-u.ac.jp}.}
	\and
	Bruno F. Louren\c{c}o\thanks{Department of Fundamental Statistical Mathematics, Institute of Statistical Mathematics, Japan.
		This author was supported partly by the JSPS Grant-in-Aid for Early-Career Scientists  23K16844 and the Grant-in-Aid for Scientific Research (B) JP21H03398.
		Email: \href{bruno@ism.ac.jp}{bruno@ism.ac.jp}.}
}

\numberwithin{equation}{section}
\begin{document}	
\maketitle
\begin{abstract}
The starting point of this paper is the computation of 
minimal hyperbolic polynomials of duals of cones arising from chordal sparsity 
patterns. From that, we investigate the relation between ranks of homogeneous cones and 
their minimal polynomials. Along the way, we answer in the negative a question posed in 
an earlier paper and show examples of homogeneous cones that cannot be realized as rank-one generated (ROG) hyperbolicity cones.\newline

\emph{Dedicated to Prof.~R.T.~Rockafellar on the occasion of his 90th birthday.}
\end{abstract}
%{\bfseries Keywords:}
%hyperbolic polynomial, hyperbolicity cone, homogeneous cone

\section{Introduction}\label{sec:int}
 Hyperbolic polynomials are intrinsically connected with many important classes of optimization problems such as linear programs, semidefinite programs and, more generally, symmetric cone programming.  %In particular, every symmetric cone is homogeneous and every homogeneous cone is a hyperbolicity cone.
%  we have the following inclusions:
%\begin{equation*}%\label{eq:cone_inclusion}
%\text{symmetric cones} \subsetneq \text{homogeneous cones} %\subsetneq \text{hyperbolicity cones}.
%\end{equation*}

We recall that if $p: \ambSpace \to \Re$ is a homogeneous polynomial of degree $d$ over a real finite-dimensional Euclidean space $\ambSpace$, then we say that it is \emph{hyperbolic along $e$} if $p(e)>0$ and for every $x \in \ambSpace$ the one-variable polynomial $t \mapsto p(x-te)$ only has real roots.
With that, the corresponding \emph{hyperbolicity cone} is 
\begin{equation}
\label{eq:hypcone}
\Lambda_+(p,e) \coloneqq \{x\in \ambSpace \mid \textup{all roots of $t\mapsto p(te-x)$ are nonnegative}\}.
\end{equation}
For example, if $p: \Re^n \to \Re$ is such that $p(x_1,\ldots, x_n) = \Pi_{i=1}^n x_i$ and $e = (1,\ldots, 1)$ then 
$\Lambda_+(p,e) = \Re^n_+$.
Similarly, over the real symmetric matrices $\cS^n$, we 
have $\Lambda_+(\det,I_n) = \psdcone{n}$, where $I_n$ is the $n\times n$ identity matrix, $\det: \cS^n \to \Re$ is the determinant polynomial and $\psdcone{n}$ is the cone of $n\times n$ real symmetric positive semidefinite matrices.

%\todo[inline]{B:Rewrite so that 2nd paragraph is the first. And try be more high-level}
%In this way,
%We recall some key definitions. A closed convex cone $\stdCone \subseteq \ambSpace$ is said to be \emph{homogeneous} if its group of automorphisms acts transitively on its relative interior, i.e., 
%$\forall x,y \in \reInt \stdCone$, there exists $A \in \Aut(\stdCone)$ such that  $Ax = y$.
%Then, a cone $\stdCone$ is said to be \emph{self-dual} if there exists some 
%inner product $\inProd{\cdot}{\cdot}$ over $\ambSpace$ for which 
%$\stdCone$ coincides with its dual $\stdCone^* \coloneqq \{y \mid \inProd{x}{y} \geq 0, \forall x \in \stdCone  \}$.
%With that a \emph{symmetric cone} is a self-dual homogeneous cone.
%Examples of symmetric cones are the positive semidefinite matrices and the nonnegative orthant.

Homogeneous cones and symmetric cones are hyperbolicity cones, but that is not the end of the story. One key point is that if a given $\stdCone$ can be written as a hyperbolicity cone $\Lambda_+(p,e)$, there could be multiple hyperbolic polynomials that generate $\stdCone$. 
For example, we have $\Re^3_+ = \Lambda_+(p^*, e)$, where 
$p^*(x_1,x_2,x_3) = x_1^2x_2x_3$ even though $\Re^3$ is also generated by the degree $3$ polynomial $x_1x_2x_3$.

Given $\stdCone$, a natural question then is how to find a minimal polynomial (i.e., a polynomial of minimal degree) that generates $\stdCone$. Fortunately, such a question is well-defined, because two minimal polynomials for the same hyperbolicity cone must differ by a constant, see \cite[Lemma~2.1]{HV07}.
Unfortunately, such a question is also notoriously hard to answer, see Section~\ref{sec:minpoly} where we overview some 
facts on this topic.

At first glance, finding minimal polynomials might seem an esoteric task, but it is actually filled with practical implications. 
G\"uler proved that $-\log p$ is a logarithmically homogeneous self-concordant barrier for the interior of $\Lambda_{+}(p,e)$ with barrier parameter $d = \deg p$, see \cite[Section~4]{Gu97}. 
By its turn, the barrier parameter can be used to upper bound the worst-case iteration complexity of interior-point methods.
The summary is that  using a polynomial of smaller degree may lead to improved complexity properties. %we can potentially make an interior point method faster by.
%Thus finding the minimal degree polynomial may have important computational implications.

While the minimal polynomials of symmetric cones are known, 
there seems to be several gaps in our knowledge of minimal polynomials of homogeneous cones and general hyperbolicity cones.

Given a homogeneous cone $\stdCone$, 
G\"uler showed in \cite[Section~8]{Gu97} how to construct a hyperbolic polynomial $p$ such 
that $\stdCone = \Lambda_{+}(p,e)$ holds for some $e$. Unfortunately, 
his construction (which was based on results of Gindikin \cite{Gi92}) does not produce minimal polynomials even when $\stdCone = \psdcone{n}$ so, the resulting self-concordant barrier may fail to be optimal with respect to the barrier parameter.
However, in \cite[Section~4.2]{CT06},  Chua and Tun\c{c}el discuss a general construction based on semidefinite representations of homogeneous cones which, in particular, recovers optimality when specialized to $\stdCone = \psdcone{n}$.
Having access to a semidefinite representation of a homogeneous cone $\stdCone$ also makes it possible to compute a hyperbolic polynomial that generates $\stdCone$ and we will revisit this point in the proof of Proposition~\ref{prop:sdp_representation}. 
In general, however, there seems to be no guarantee that such polynomial is minimal.

As mentioned previously, part of the difficulty is that it is hard to attest that a given hyperbolic polynomial is  minimal. 
However, there is one important exception to this: the so-called \emph{rank-one generated (ROG) hyperbolicity cones}, which were studied in \cite{IL23}.
A hyperbolicity cone $\Lambda_{+}(p,e)$ is ROG if its extreme rays are generated by rank-$1$ elements, where the rank of a given $x$ is the number of nonzero roots  of the polynomial $t\mapsto p(x - te)$.

If a pointed full-dimensional hyperbolicity cone $\Lambda_{+}(p,e)$ is ROG (with respect to $p$) then $p$ must be a minimal polynomial for $\Lambda_{+}(p,e)$, see \cite[Proposition~3.5]{IL23}. 
%This helps bring some geometry in what was an otherwise purely algebraic question.
This helps bring some geometry into what was an otherwise purely algebraic question.
With that, the following question was formulated in \cite[Section~5]{IL23}.
\begin{quotation}
	Can homogeneous cones be realized as ROG hyperbolicity cones?
\end{quotation}
Unfortunately, as we will see in this paper the answer is \emph{no}. 
We will exhibit a homogeneous cone that cannot be realized as a ROG hyperbolicity cone. 

We also recall that there is a notion of rank for homogeneous cone. 
Although it is known that a symmetric cone has a minimal polynomial equal to its rank, we will also observe that this is no longer true for homogeneous cones. 
Indeed, there exists a homogeneous cone of rank $3$ whose minimal polynomial is $4$, see Section~\ref{sec:hrank}. More generally, we will present a result on the degrees of the minimal polynomials for cones arising as duals of PSD cones with chordal sparsity patterns.

Here is a summary of our results.
\begin{itemize}
	\item We provide minimal polynomials for the duals of PSD cones with chordal sparsity patterns, see Theorem~\ref{th:minpoly-chordal}. Then, based on this result, we furnish several consequences for hyperbolicity and homogeneous cones. 
	Among the PSD cones with chordal sparsity patterns, if an additional condition is satisfied, these cones become homogeneous. In combination with 
	Theorem~\ref{th:minpoly-chordal}, this allow us to find several examples of homogeneous cones that cannot be realized as ROG hyperbolicity cones, see Theorem~\ref{theo:chord_dual_rog} and Section~\ref{sec:hcone}.
	
	\item We present a discussion on the connection of homogenous cone rank and the degree of minimal polynomials for homogeneous cones. 
	In particular, if the $\stdCone$ is the dual of a cone that arises from a homogeneous chordal sparsity pattern, then the degree of minimal polynomials is no more than $O(n^2)$ where $n$ is the homogeneous cone rank.
	We also present some results for general homogeneous cones, but they are more limited in scope. %For example, we discuss an exponential bound for general homogeneous cones, see Proposition~\ref{prop:exp_bound}.
For example, we discuss an exponential upper bound for the degree of minimal polynomials for a general homogeneous cone in terms of its rank, see Proposition~\ref{prop:exp_bound}.	
\end{itemize}

This work is organized as follows. In Section~\ref{sec:prel} we recall some basic facts about semialgebraic and hyperbolicity cones. 
In Section~\ref{sec:chordal}, we present a discussion on the dual of cones arising from chordal sparsity patterns. Finally, in Section~\ref{sec:hyp_hcone} we have results on ROG hyperbolicity cones and minimal polynomials of homogeneous cones.

\section{Preliminaries}\label{sec:prel}
We recall some basic nomenclature on convex sets that will be needed in the sequel, see \cite{RT97} for more details.
Let $\stdCone$ be a closed convex cone $\stdCone$ contained in a finite dimensional Euclidean space $\ambSpace$. $\stdCone$ is said to be \emph{pointed} if $\stdCone \cap - \stdCone = \{0\}$ and it is said to be \emph{full-dimensional} if the linear span of $\stdCone$ is $\ambSpace$. We denote the relative interior of $\stdCone$ by $\reInt \stdCone$. 

\subsection{Semialgebraic cones and minimal polynomials}\label{sec:minpoly}

We say a set $S \subseteq \Re^n$ is \emph{semialgebraic} if it can be written as a finite union of sets of the form
$$\{x \in \Re^n~|~p_i(x)\geq 0~(i=1,\ldots,k),~q_j(x)>0~(j=1,\ldots,l)\},$$
for some polynomials $p_i(x),q_j(x)$ of real coefficients.
We will refer by \emph{semialgebraic cone} to any convex cone that is also a semialgebraic set. 

\begin{example}
In convex analysis, this class includes many important examples: the nonnegative orthant $\RR_+^n$, the cone of $n\times n$ real symmetric positive semidefinite matrices $\psdcone{n}$, the copositive cone $\textup{CoP}^n$ and the completely positive cone $\textup{CP}^n$, second order cones, hyperbolicity cones, etc.
A non-example would be the power cone ${{{\cal P}^{^{\bm{\alpha}}}_{m,n}}}$ for irrational $\alpha$ and $n \geq 2$ \cite{Cha09,LLLP23}.
\end{example}

Semialgebraic cones are very special semialgebraic sets, that come with some particularly useful algebraic properties. Most importantly we have the following fact.

\begin{proposition}\label{prop:bound_hyper}
Let $C$ be a closed semialgebraic cone with non-empty interior. Then the algebraic closure of its boundary is a hypersurface. This means that there exists (up to multiplication by a scalar) a unique polynomial of minimal degree that vanishes on the boundary of $C$, and any other polynomial vanishing on the boundary must be a multiple of it.
\end{proposition}

The statement of Proposition \ref{prop:bound_hyper} is in fact valid for any semialgebraic set $S$ such that $S$ and its complement are \emph{regular} and non-empty. Here, a regular set is one that is contained in the closure of its interior. For a proof see Lemma~2.5 of \cite{sinn2015algebraic}. We refer to that paper for a thorough survey on the algebraic theory of boundaries of convex sets.
In particular, a convex semialgebraic set $C \subsetneq \Re^n$ with non-empty interior and its complement are always regular.   \medskip

The unique polynomial (up to scaling) that Proposition \ref{prop:bound_hyper} associates to a   closed  semialgebraic cone  $C$ with non-empty interior is called the \emph{minimal polynomial} of $C$.

\begin{example} Examples of minimal polynomials of semialgebraic cones.
\begin{enumerate}

\item For $C=\RR_+^n$ the minimal polynomial is $x_1x_2...x_n$. It vanishes on the boundary and any polynomial that divides it does not.

\item For $C=\psdcone{n}$ the minimal polynomial is $\det(X)$. To see this it is enough to note that it vanishes on the boundary and that it is irreducible, so it cannot be a multiple of any other polynomial.

\item For $C=\textup{CP}^n$ the minimal polynomial is very hard to compute. In \cite{pfeffer2021cone} the case $n=5$ is partially answered. In particular a divisor of the minimal polynomial is computed and it has degree 3900, see Theorem~2.13 therein.
 \end{enumerate}  
\end{example}

The next lemma will be useful to argue about minimal polynomials of intersections of cones.
%added:
Here, we say that a polynomial $p$ is a \emph{strict divisor} of a polynomial $q$ if $q = ph$ holds for some polynomial $h$ of degree $\geq 1$.

\begin{lemma}\label{lemma:minpoly-prod}
Let $C_1,\ldots,C_k \subsetneq \Re^n$ be closed semialgebraic cones with minimal polynomials $p_1,\ldots,p_k$, respectively. Suppose that the interiors of the $C_i$ intersect and the $p_1,\ldots,p_k$ are irreducible. 
Then, $p_1\cdots p_k$ is a minimal polynomial of $C_1\cap \cdots \cap C_k$ if and only if any strict divisor of $p_1\cdots p_k$ does not vanish entirely on $C_1\cap \cdots \cap C_k$.
\end{lemma}

\begin{proof}
The necessity is clear by the minimality of $p_1\cdots p_k$. To show the sufficiency, let $q$ be a minimal polynomial of $C_1\cap \cdots \cap C_k$. Since $p_1\cdots p_k$ vanishes on the boundary of $C_1\cap \cdots \cap C_k$, $q$ divides $p_1\cdots p_k$ (Proposition~\ref{prop:bound_hyper}). As $\RR[x]$ is a unique factorization domain and the factors $p_i$ are irreducible, $q$ has the expression $q = \alpha \prod_{j \in J} p_j$ for some $\alpha  \in \RR$ and $J\subset\{1,\ldots,k\}$. But $J=\{1,\ldots,k\}$ must hold otherwise $q$ becomes a strict divisor of $p_1\cdots p_k$.
\end{proof}

\subsection{Hyperbolicity cones and their minimal polynomials}

An important case of semialgebraic cones is that of hyperbolic cones. Given $p:\ambSpace\to\Re$ a homogeneous polynomial of degree $d$ that is hyperbolic along a direction $e \in \ambSpace$, we defined its hyperbolicity cone (along $e$) in \eqref{eq:hypcone}.
Here, a polynomial over $\ambSpace$ is understood as a polynomial over the coefficients of some basis of $\ambSpace$.

From that definition is not clear if a hyperbolicity cone is even a convex cone, and it is a foundational result of G{\aa}rding \cite{Gard51} that this is in fact the case. That it is semialgebraic follows, for example, from \cite[Theorem 20]{Re06} that states that the hyperbolicity cone is cut out by $p$ and its directional derivatives with respect to $e$. 

Another related very useful description of hyperbolicity cones is the following basic fact (see, e.g., \cite{Gard51} and \cite[Proposition 1]{Re06}).

\begin{lemma} \label{lem:bound_description}
The hyperbolicity cone $\Lambda_+(p,e)$ is the closure of the connected component of $\{x \mid p(x) \not= 0\}$ containing $e$.
\end{lemma}
We note that  $e$ is in the interior of $\Lambda_+(p,e)$, so 
$\Lambda_{+}(p,e)$ is always full-dimensional in $\ambSpace$.
%Suppose additionally that $\Lambda_+(p,e)$ is pointed.
Lemma~\ref{lem:bound_description} immediately guarantees that $p$ vanishes at the boundary of $\Lambda_+(p,e)$, so the minimal polynomial of the cone, whose existence is guaranteed by Proposition~\ref{prop:bound_hyper} (see also \cite[Lemma~2.1]{HV07}), must divide $p$ over $\mathcal{V}$.
This in turn gives us the following well-known fact.

\begin{lemma}\label{lem:min_hyp}
A minimal polynomial $q:\mathcal{V} \to \Re$ of a hyperbolicity cone $\Lambda_+(p,e)$ is hyperbolic with respect to direction $e$ and $\Lambda_+(p,e)=\Lambda_+(q(e)q,e)$.
\end{lemma}
\begin{proof}
By the above observation, $p=qr$ for some polynomial $r$. This implies that for $x \in \mathcal{V}$ we have $p(te-x)=q(te-x)r(te-x)$, and since the roots of $p(te-x)$ are all real so are the ones of $q(te-x)$. By fixing the sign so that $q$ is non-negative in $\Lambda_+(p,e)$, we conclude that $\Lambda_+(p,e) \subseteq \Lambda_+(q(e)q,e)$
and that $\Lambda_+(q(e)q,e)$ is in the connected component of $\{x \mid p(x) \not= 0\}$ containing $e$.
By Lemma \ref{lem:bound_description} we conclude that $\Lambda_+(p,e)=\Lambda_+(q(e)q,e)$.
\end{proof} 
%For a general hyperbolicity cone $\Lambda_{+}(p,e)$ Lemma~\ref{lem:min_hyp}, a polynomial for a pointed 
The minimal polynomials for  a hyperbolicity cone $\Lambda_+(p,e)$  are exactly the hyperbolic polynomials  $\hat p$ of minimal degree for which 
$\Lambda_{+}(p,e) = \Lambda_{+}(\hat p, e)$ holds.
We note that minimality as just described and minimality in the sense of Section~\ref{sec:minpoly} coincide by Lemma~\ref{lem:min_hyp}.
%This follows from, say, Lemma~\ref{lem:bound_description} which implies that hyperbolicity cones are \emph{algebraic interiors} and \cite[Lemma~2.1]{HV07}.
%In particular,  if $p$ is minimal and $\Lambda_{+}(p,e) = \Lambda_{+}(q,e)$ holds, then $p$ must be a divisor of $q$. 
%For pointed hyperbolicity cones, minimality as just described and minimality in the sense of Section~\ref{sec:minpoly} coincides by Lemma~\ref{lem:min_hyp}.
From the discussion so far can now specialize Lemma~\ref{lemma:minpoly-prod} to the case of hyperbolicity cones.

\begin{proposition} \label{prop:minpoly-prod-hyp}
Let $p_1,\ldots,p_k:\ambSpace\to\Re$ be hyperbolic polynomials along $e \in \ambSpace$.
Assume that $p_1,\ldots,p_k$ are irreducible.
Then $p_1\cdots p_k$ is a minimal polynomial for $\Lambda_+(p_1\cdots p_k,e)$ if and only if
\begin{equation*}%\label{eq:hyp-prod-minpoly-cond}
\quad \forall j=1,\ldots,k,~~\bigcap_{\substack{i=1 \\ i\neq j}}^k \Lambda_+(p_i,e) \supsetneq \Lambda_+(p_1\cdots p_k,e).
\end{equation*}
\end{proposition}
\begin{proof}
Observing that $\bigcap_{i=1}^k \Lambda_+(p_i,e) = \Lambda_{+}(p_1\cdots p_k,e)$ holds, the proof follows  Lemma~\ref{lemma:minpoly-prod}.\end{proof}

\section{Duals of chordal cones and minimal polynomials}\label{sec:chordal}

Of special interest to us is the example of \emph{chordal cones} and their duals. Given a graph $G$ with nodes $V=\{1,...,n\}$ and edges $E$, we
consider the subspace 
\begin{equation}\label{eq:sg}
\mathcal{S}(G) \coloneqq\{X \in  \mathcal{S}^n\, | \, X_{ij}=0 \textrm{ for all } i\not=j \textrm{ such that } \{i,j\} \not \in E\}
\end{equation}
and the corresponding cone of positive semidefinite matrices that respect the sparsity pattern induced by $G$ %given by 
\begin{equation}\label{eq:spg}
\mathcal{S}_+(G)\coloneqq \mathcal{S}(G) \cap \psdcone{n}. %=  \{X \in  \mathcal{S}^n_+ \, | \, X_{ij}=0 \textrm{ for all } i\not=j \textrm{ such that } \{i,j\} \not \in E\}.
\end{equation}
%Denote by $S(G)$ the set of symmetric matrices with the same support, i.e., the span of $S_+(G)$. 
A cone that is related to $\mathcal{S}_+(G)$ is 
$$\mathcal{S}^*(G)=\{ A \in \mathcal{S}(G) \, | \, A_C \succeq 0 \textrm{ for all } G\textrm{-cliques } C\},$$
where $A_C$ is the submatrix indexed by the vertices of $C$.
The cone $\mathcal{S}^*(G)$ is a hyperbolicity cone since it is the intersection of the symmetric cones $\{A \in \mathcal{S}(G)~|~A_C\succeq 0\}$ among $G$-cliques $C$.
By the classic result of \cite[Theorem 7]{grone1984positive} we have  $\mathcal{S}^*(G) = \mathcal{S}_+(G)^{*}$ if and only if $G$ is chordal, in which case we can also say by \cite[Theorem 3.3]{agler1988positive} that
$$\mathcal{S}_+(G)=\left\{ \sum_i A_i \, | \, A_i \succeq 0 \textrm{ with } \textup{supp}(A_i) \textrm{ contained in a clique of } G\right\}.$$
We observe that the interior of $\mathcal{S}^*(G)$ is given by 
\begin{equation}\label{eq:interior}
\interior(\mathcal{S}^*(G)) = \{ A \in \mathcal{S}(G) \, | \, A_C \succ 0 \textrm{ for all maximal  } G\textrm{-cliques } C\}.
\end{equation}
Let $I_n$ denote the $n\times n$ identity matrix, \eqref{eq:interior} holds
because $(I_n)_C \succ 0$ for all $G$-cliques $C$, so $\interior(\mathcal{S}^*(G)) = \{ A \in \mathcal{S}(G) \, | \, A_C \succ 0 \textrm{ for all } G\textrm{-cliques } C\}$, by \cite[Theorem~6.5]{RT97}. Since every $G$-clique is contained in some maximal clique and principal submatrices of positive definite matrices are positive definite, we obtain \eqref{eq:interior}.

%because this is an open subset of $S^*(G)$ and when we have $A$ in $S^*(G)$ with some $A_C$ singular we can always perturb by an arbitrarily small matrix in $S(G)$ to make one of its eigenvalues negative, so the remainder of the $S^*(G)$ is in the boundary.

\begin{theorem}\label{th:minpoly-chordal}
A minimal polynomial of the hyperbolicity cone $\mathcal{S}^*(G)$, for $G$ chordal, is
$$p_G(X)=\prod_{C \textrm{ maximal clique of }G} \det(X_C),$$
where $X$ is a symmetric matrix of unknowns.
\end{theorem}

\begin{proof}
%\todo[inline]{M: I will try to replace Lemma \ref{lemma:minpoly-prod} with Proposition \ref{prop:minpoly-prod-hyp} as $K_C$ can be non-pointed.}
%If there exists a single maximal clique of $G$, then $S_+(G) = \psdcone{n} = S^*(G)$ and the result is true. In what follows, we will assume that there exists at least two maximal cliques.
For a clique $C$ of $G$ consider the cone $K_C=\{X \, | \, X_C \succeq 0\}$ and its minimal polynomial $p_C=\det(X_C)$. %$K_C$ is full dimensional and the $n\times n$ identity matrix $I_n$ is in its interior. 
With that, we have $K_C = \Lambda_{+}(p_C,I_n)$ and $\mathcal{S}^*(G)=\bigcap_{C} K_C = 
\Lambda_{+}(p_G,I_n)$, where $C$ varies over the maximal cliques of $G$. By Proposition~\ref{prop:minpoly-prod-hyp}, %we just have to show that no strict divisor of $\prod_C p_C$ vanishes at the boundary of $S^*(G)$. 
since every such $p_C$ is irreducible, we just have to show that none of the factors can be omitted. 
In order to do so, for any maximal clique $C$ we construct a matrix $A$ in $\mathcal{S}^*(G)$ such that  $A_C$ is singular (i.e., is in the boundary of $\mathcal{S}^*(G)$) but all the other maximal clique-indexed submatrices are positive definite (i.e., $A$ is in the interior of the intersection of the $K_D$ for $D \neq C$).

To do this, start by constructing $A_C$ to be positive semidefinite and singular but such that all its proper principal submatrices are positive definite. For instance let it be the matrix with $|C|-1$ on the diagonal and $-1$ on every other entry. Complete $A$ by adding $1$ to every other diagonal entry not in $C$. For any maximal clique $D \not = C$ of $G$, the matrix $A_D$ will be block diagonal, with one block being $A_{D \cap C}$ and the other being an identity. Since both $C$ and $D$ are maximal cliques, $D \cap C$ is strictly contained in $C$, so $A_{D \cap C}$ is a proper principal submatrix of $A_C$ hence positive definite. This implies that $A_D$ is positive definite for any maximal clique other than $C$. In particular the only term of the product vanishing at $A$ is $p_C(X)$ so we cannot drop it.
\end{proof}

It follows that the degree of the minimal polynomial of $\mathcal{S}^*(G)$ is the sum of the sizes of its maximal cliques.
While general graphs can have exponentially many cliques with respect to the number of vertices, chordal graphs admit no such behavior. 
In fact, since every chordal graph has an elimination order, there are at most $n$ maximal cliques in a chordal graph with $n$ vertices (see \cite{chordal} for more details on chordal graphs). Since the size of each clique has a trivial bound of $n$, this tells us that the degree of the minimal polynomial of $\mathcal{S}^*(G)$ is $O(|V(G)|^2)$, where $V(G)$ is the set of vertices of $G$. While $|V(G)|^2$ can never be attained, the degree can in fact grow quadratically as illustrated in the following example.

\begin{example}\label{ex:max_degree}
For each $n$ consider the graph $G_n$ constructed in the following way: we take a clique $K$ of size $\lfloor \frac{n-1}{2} \rfloor$ and add $\lceil \frac{n+1}{2} \rceil$ new vertices each connecting to every vertex of $K$ and nowhere else. This graph has $n$ vertices and $\lceil \frac{n+1}{2} \rceil$ maximal cliques of size $\lfloor \frac{n-1}{2} \rfloor+1=\lfloor \frac{n+1}{2} \rfloor$. 

Using Theorem \ref{th:minpoly-chordal} we conclude that the degree of the minimal polynomial of $\mathcal{S}^*(G_n)$ is $\frac{(n+1)^2}{4}$ for $n$ odd and $\frac{(n+1)^2-1}{4}$ if $n$ even. 

\end{example}

\section{Consequences for hyperbolicity and homogenous cones}\label{sec:hyp_hcone}
Before we proceed, we recall that a closed convex cone $\stdCone \subseteq \ambSpace$ is said to be \emph{homogeneous} if its group of automorphisms acts transitively on its relative interior, i.e., 
$\forall x,y \in \reInt \stdCone$, there exists $A \in \Aut(\stdCone)$ such that  $Ax = y$.
Then, a cone $\stdCone$ is said to be \emph{self-dual} if there exists some 
inner product $\inProd{\cdot}{\cdot}$ over $\ambSpace$ for which 
$\stdCone$ coincides with its dual cone $\stdCone^* \coloneqq \{y \mid \inProd{x}{y} \geq 0, \forall x \in \stdCone  \}$.
With that a \emph{symmetric cone} is a self-dual homogeneous cone.
Examples of symmetric cones are the $n\times n$ positive semidefinite matrices over the real numbers $\psdcone{n}$, the nonnegative orthant $\Re^n_+$ and the second-order cone $\SOC{2}{n+1} \coloneqq \{(x,t) \mid t \geq 
\norm{x}_2 \}$, where $\norm{\cdot}_2$ indicates the $2$-norm.

Symmetric cones have a natural notion of rank that comes from its Jordan algebraic structure. Fortunately, we will not need to discuss the details of that and it suffices to recall that the ranks of $\psdcone{n}$, $\Re^n_+$ and $\SOC{2}{n+1}$ are $n$, $n$ and $2$, respectively.

We also recall that given a hyperbolic polynomial $p:\ambSpace \to \Re$ along a hyperbolicity direction $e$, we define the rank of $x \in \ambSpace$ as
\begin{equation*}\label{eq:rank_def}
\rank(x) \coloneqq \text{number of nonzero roots of } t \mapsto p(te-x).
\end{equation*}
A hyperbolicity cone $\Lambda_{+}(p,e)$ is said to be \emph{rank-one generated (ROG)} if its extreme rays are generated by rank-$1$ elements. 
More generally, a given closed convex cone $\stdCone$ is said to be \emph{realizable as  a ROG hyperbolicity cone} if there exists $p$ and $e$ such that $\stdCone = \Lambda_{+}(p,e)$ and  $\Lambda_{+}(p,e)$ is ROG.
The subtlety here is that even if $\Lambda_{+}(p,e) = \Lambda_{+}(\hat p,e)$ holds for a certain cone, the rank function induced by $p$ and $\hat p$ may still be different, so ROGness depends on the choice of $p$.
However, if  $\Lambda_{+}(p,e)$ is ROG, then $p$ must be a minimal polynomial for $\Lambda_{+}(p,e)$, see \cite[Proposition~3.5]{IL23}. 
%So, in essence, in order to decide whether a given a closed convex cone is realizable as a ROG hyperbolicity cone it is enough to examine one of its minimal degree polynomials. ROG or not

All symmetric cones are realizable as ROG hyperbolicity cones, for more details see \cite[proof of Lemma~4.1]{GT98} or \cite[Section~3.1.1 and Proposition~3.8]{IL23}. 
For example, we have $\psdcone{n} = \Lambda_{+}(\det,I_n)$. The extreme rays of $\psdcone{n}$ correspond to rank-$1$ matrices (in the usual linear algebraic sense). They  also correspond to 
rank-$1$ elements in the hyperbolic sense, since the matrix rank of $X \in \psdcone{n}$ is the number of nonzero roots of $t \mapsto \det(tI_n-X)$.
In this way, a minimal polynomial of $\psdcone{n}$ is indeed the determinant polynomial, which has degree $n$. For $\Re^n_+$ and $\SOC{2}{n+1}$, minimal polynomials have degrees $n$ and $2$, respectively.
More generally,  a symmetric cone has a minimal polynomial of degree equal to its rank\footnote{The precise proof of this statement comes from the fact a symmetric cone admits a generalized notion of determinant, which is a polynomial of degree equal to its rank. 
With respect to this determinant, a symmetric cone can be realized as a ROG hyperbolicity cone, see the discussion in \cite[Section~3.1.1]{IL23}.}.

If $G$ is chordal, then $\mathcal{S}_+(G)$ (see \eqref{eq:spg}) is a ROG hyperbolicity cone with respect to the usual determinant polynomial and $I_n$, since every $X \in \mathcal{S}_+({G})$ has a decomposition in rank-$1$ matrices that respects the sparsity pattern defined by $G$.
$\mathcal{S}^*(G)$ is also a hyperbolicity cone, but, in contrast, it is rarely ROG as we will see in the next result.

\begin{theorem}\label{theo:chord_dual_rog}
Let $G$ be a chordal graph.
$\mathcal{S}^*(G)$ is realizable as a ROG hyperbolicity cone if and only if 
$G$ is a disjoint union of cliques.
\end{theorem}
\begin{proof}
If $G$ is a disjoint union of cliques $G_1, \ldots G_m$, then $\mathcal{S}^*(G)$ is linearly isomorphic to 
a direct product of smaller positive semidefinite cones that are supported on the 
$G_i$. In particular, 	$\mathcal{S}^*(G)$ is a symmetric cone and therefore it is realizable as a ROG hyperbolicity cone.
	
Conversely, suppose that $G$ is not a disjoint union of cliques.
 Then, there exists a least one node of $G$ that is contained in two distinct maximal cliques  $G_1,G_2$ of $G$.	
Relabelling the nodes if necessary, we may assume that this node is $1$. 
Let $E$ be the matrix that is $1$ at the $(1,1)$-entry and zero elsewhere. 
First, we will prove that $E$ generates an extreme ray of $\mathcal{S}^*(G)$. 
Let $X,Y \in \mathcal{S}^*(G)$ be nonzero and assume that $X+Y = \alpha E$ for some $\alpha \geq 0$. Then, for every $G$-clique $C$ we have
\[
X_C + Y_C = \alpha E_C.
\]
If $1$ is not in the node set of $C$, then $E_{C}$ is the zero matrix, so 
$X_C + Y _C = 0$, which in view of $X_C \succeq 0, Y_C \succeq 0$, implies 
$X_C = Y_C = 0 = E_C$. If $1$ is in the node set of $C$, then, assuming that $C$ has $k$ nodes,  $E_C$ generates an extreme ray of $\psdcone{k}$. 
Therefore,   $X_C \succeq 0, Y_C \succeq 0$ and $X_C + Y_C = \alpha E_C$ implies 
that $X_{C}$ and $Y_C$ are positive multiples of $E_C$.
The conclusion is that for every $G$-clique $C$, the entries of $X_C, Y_C$ are zero with the potential exception of the $(1,1)$-entry if $C$ contains $1$.
We note that every pair of indices $(i,j)$ with $i\neq j$ is either outside the edge set of $G$ (in which case of $X_{ij} = Y_{ij} = 0$) or is in $G$, in which case, it is covered by some clique of $G$.
Overall, we obtain that  $X_{11}$ and $Y_{11}$ are positive and all the other entries of $X$ and $Y$ are zero. This shows that $E$ is indeed an extreme ray of $\mathcal{S}^*(G)$.

The next step is showing that no matter which hyperbolic polynomial $q$
is chosen for $\mathcal{S}^*(G)$, the extreme ray $E$ will always have rank greater than $1$.
Suppose that $\mathcal{S}^*(G) = \Lambda_{+}(q,U)$, where $q$ is a hyperbolic polynomial along the direction $U$. % with respect to $q$. Then, $q$ must be a minimal degree polynomial for $S_+(G)^*$, by
%\cite[Proposition~3.5]{IL23}. 
Since the polynomial given in Theorem~\ref{th:minpoly-chordal} is 
minimal, there exists some polynomial $h$ such that 
\begin{equation}\label{eq:qh}
q(X) = h(X)\prod_{C \textrm{ maximal clique of }G} \det(X_C)
\end{equation}
holds, which follows from either Proposition~\ref{prop:bound_hyper} or \cite[Lemma~2.1]{HV07}.
Now we recall some facts about hyperbolicity cones. 
First, if $\Lambda_{+}(p, \bar{e})$ is a hyperbolicity cone, 
then for any $\hat e$ belonging to the relative interior of $\Lambda_{+}(p,\bar{e})$, we 
have $\Lambda_{+}(p,\hat e) =  \Lambda_{+}(p,e)$, see \cite[Theorem~3]{Re06}. Also,  
the rank of $x \in \Lambda_{+}(p, \bar{e})$ with respect to $p$ and $\bar{e}$ (i.e., the number of roots of $t \mapsto p(t\bar{e}-x) $)
coincides with the rank of $x$ with respect to $p$ and $\hat e$ (i.e., the number of roots of $t \mapsto p(t\hat{e}-x) $), see \cite[Proposition~22]{Re06}\footnote{A minor detail is that Renegar's result is about the multiplicity of zero as a root of $t \mapsto p(te-x)$.}. 

In particular, the (hyperbolic) rank of $E$ computed with respect $q$ and $U$ is the same as 
the rank of $E$ computed with respect to $q$ and the identity matrix $I_n$.
By assumption, the node $1$ is in at least two maximal cliques $G_1, G_2$ of $G$, so the rank of $E$ is at least two because of the terms $\det(X_{G_1})$ and $\det(X_{G_2})$ in \eqref{eq:qh}.	
\end{proof}

\subsection{Consequences for homogeneous cones}\label{sec:hcone}
Among the chordal cones described in Section~\ref{sec:chordal}, there is special class of cones that are  homogeneous. 
These are the cones $\mathcal{S}_+(G)$ (see \eqref{eq:spg}) where $G$ not only is chordal but also does not contain induced subgraphs isomorphic to a $4$-node path. 
We will refer to such a graph $G$ as being a \emph{homogeneous chordal graph}.
For a proof that $\mathcal{S}_+(G)$ is homogeneous see \cite[Theorem~A]{Ishi13}.
This class of cones was also extensively studied in \cite{TV23}.
The cones  $\mathcal{S}^*(G)$ are also homogeneous because duality preserves homogeneity.

As mentioned in Section~\ref{sec:int}, an open question that was left in \cite{IL23} was whether homogeneous cones can be realized as ROG hyperbolicity cones. 
Theorem~\ref{theo:chord_dual_rog} immediately implies that the answer is ``No'', since $G$ may be homogeneous chordal without being a disjoint union of cliques, in which case $\mathcal{S}^*(G)$ will fail to be ROG. 

For the sake of concreteness, we present a detailed example. 
Consider the following graph $G$.
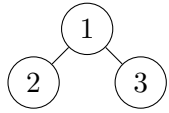
\begin{center}
	\begin{tikzpicture}[main/.style = {draw, circle}] 
	\node[main] (1) {$1$}; 
	\node[main] (2) [below left of =1]{$2$}; 
	\node[main] (3) [below right of =1]{$3$}; 
	\draw (1) -- (2);
	\draw (1) -- (3);
	\end{tikzpicture}
\end{center}
With that, $\mathcal{S}_{+}(G)$ can be described as follows
\[
\mathcal{S}_{+}(G) = \left\{(x_1,x_2,x_3,x_4,x_5) \in \mathbb{R}^5 \mid \begin{pmatrix}
x_1 & x_2 & x_3\\
x_2 & x_4 & 0\\
x_3 & 0   & x_5\\
\end{pmatrix} \succeq 0 \right\}.
\]
The graph $G$ is chordal and since it has only $3$ vertices it is also homogenous chordal thus, $\mathcal{S}_{+}(G)$ is a homogeneous cone of dimension $5$. We equip $\mathbb{R}^5$ with the Frobenius inner product $\inProd{\cdot}{\cdot}_F$ in such a way that for $x,y \in \Re^5$ we have
\[
\inProd{x}{y}_F = \tr\left(\begin{pmatrix}
x_1 & x_2 & x_3\\
x_2 & x_4 & 0\\
x_3 & 0   & x_5\\
\end{pmatrix} \begin{pmatrix}
y_1 & y_2 & y_3\\
y_2 & y_4 & 0\\
y_3 & 0   & y_5\\
\end{pmatrix} \right).
\]
In this way, following the discussion in Section~\ref{sec:chordal}, the dual of $\mathcal{S}_{+}(G)$ is given by 
\begin{equation}\label{eq:vin_cone}
\mathcal{S}^*(G) = \left\{(y_1,y_2,y_3,y_4,y_5) \in \mathbb{R}^5 \mid \begin{pmatrix}
y_1 & y_2 \\
y_2 & y_4 
\end{pmatrix} \succeq 0, \begin{pmatrix}
y_1 & y_3 \\
y_3 & y_5 
\end{pmatrix} \succeq 0  \right\}.
\end{equation}
This cone $\mathcal{S}^*(G)$ is also called the \emph{Vinberg cone} \cite{Ishi01}.
$G$ has two maximal cliques, so it follows from Theorem~\ref{th:minpoly-chordal} that a minimal polynomial for $\mathcal{S}^*(G)$ is given by 
\[
\det \begin{pmatrix}
y_1 & y_2 \\
y_2 & y_4 
\end{pmatrix}  \det \begin{pmatrix}
y_1 & y_3 \\
y_3 & y_5 
\end{pmatrix}.
\]
In addition, by Theorem~\ref{theo:chord_dual_rog}, $\mathcal{S}_{+}^*(G)$ is not realizable as a ROG hyperbolicity cone. We note this is a corollary.

\begin{corollary}
The Vinberg cone \eqref{eq:vin_cone} is not realizable as a ROG hyperbolicity cone.
\end{corollary}
	
\subsection{Homogeneous cone rank and minimal polynomials}\label{sec:hrank}
Regarding homogeneous cones we are still left with the question of identifying  minimal polynomials and their degrees. As mentioned previously, symmetric cones have a notion of rank and their minimal polynomials have degree equal to the rank, see \cite[Proposition~3.8]{IL23}. 
So, for symmetric cones, this question is completely settled. 
A natural question then is whether the same is true for homogeneous cones. 
As we will see, the answer is no. 
But, first, we need to recall some aspects of the theory of homogeneous cones.

Given a pointed full-dimensional homogeneous closed convex cone $\stdCone$ contained in some finite dimensional Euclidean space $\ambSpace$, it is possible to associate to $\stdCone$ an algebra of ``generalized matrices'' $\mathcal{M}$ in such a way that $\stdCone$ correspond to the elements that have a ``generalized Cholesky decomposition'', i.e., 
\begin{equation}\label{eq:cholesky}
x \in \stdCone \Leftrightarrow \exists t,\,\, \text{s.t.}\,\, x = tt^*,
\end{equation}
where $t \in \mathcal{M}$ is an ``upper triangular'' generalized matrix and ``$*$'' is an operator analogous to the usual matrix adjoint. The elements of $\mathcal{M}$ are generalized matrices in the sense that each $a \in \mathcal{M}$ can be written in a unique way  as $a = \sum _{1\leq i,j \leq n} a_{ij}$, where each $a_{ij}$ belongs to certain subspaces $\mathcal{M}_{ij} \subseteq \mathcal{M}$ satisfying $\mathcal{M} = \bigoplus_{1\leq i,j \leq n} \mathcal{M}_{ij}$. This direct sum decomposition of $\mathcal{M}$ is called a \emph{bigradation}.
The subspaces $\mathcal{M}_{ii}$ in the ``diagonal'' must all be one-dimensional. However, contrary to an usual matrix algebra over a field, the $\mathcal{M}_{ij}$ with $i \neq j$ may have dimension greater than $1$ or it may even be zero.
In this way, the $t$ in \eqref{eq:cholesky} is ``upper triangular'' in the sense that $t_{ij} = 0$ if $i > j$.
Finally, by definition, the number $n$ is the rank of the algebra and corresponds to the \emph{(homogeneous) rank} of the cone $\stdCone$, see \cite[Definition~4]{Ch09}.

This $\mathcal{M}$ is a so-called \emph{T-algebra} introduced by Vinberg \cite{V63}, see also 
the works of Chua for a discussion on T-algebras focused on optimization aspects \cite{CH03,Ch09}. An earlier related discussion can also be seen in \cite[Section~8]{Gu97}.
The precise definition of T-algebra is relatively involved and describes the behavior of the bigradation of $\mathcal{M}$ with respect to upper triangular elements, the underlying algebra multiplication and the adjoint operator. 
Fortunately, for what follows we only need two facts. 
The first is that the rank of a homogeneous cone is well-defined, i.e., a cone may be realizable in multiple ways using different T-algebras, but all of them correspond to isomorphic T-algebras of same rank, see, for example, \cite[Theorem~1]{Ch09}. 
The second is that the dual of a homogeneous cone is also a homogeneous cone of the same rank. We note this as a lemma.
\begin{lemma}\label{lem:hom_cone_rank}
The dual of a homogeneous convex cone of rank $n$ also has rank $n$.
\end{lemma}
\begin{proof}
We only present a sketch of this well-known fact. A T-algebra for the dual cone is obtained from a T-algebra of the primal cone by exchanging the order of the subspaces ${\mathcal{M}}_{ij}$ in an appropriate way. This is referred to in the literature as a ``change of grading'', e.g., see \cite[Chapter 3, section~6]{V63}  or see \cite[Section~2.3]{Ch09}. In particular, the rank of the T-algebra stays the same.
\end{proof}

Returning to the issue of minimal polynomials, it turns out that, in contrast to symmetric cones, the homogeneous cone rank can be smaller than the degree of minimal polynomials.
\begin{corollary}
Let $G$ be a homogenous chordal graph on $n$ nodes. Then, $\mathcal{S}^*(G)$ has homogeneous cone rank $n$. Furthermore,  minimal polynomials  for $\mathcal{S}^*(G)$ have degree $n$ if and only if $G$ is a disjoint union of cliques.
\end{corollary}
\begin{proof}
The cone $\mathcal{S}_+(G)$ is a homogeneous cone of rank $n$, so its dual $\mathcal{S}^*(G)$ is also homogeneous of rank $n$, by Lemma~\ref{lem:hom_cone_rank}.	
Each polynomial $\det (X_C)$ in Theorem~\ref{th:minpoly-chordal} has degree equal to the number of nodes of $C$. The only way that $p_G$ can have degree $n$ is if each node is in at most one maximal clique.	
\end{proof}
In particular, the Vinberg cone in \eqref{eq:vin_cone} has rank $3$ but its minimal polynomials have degree $4$.

A natural question then becomes how to bound the degree of minimal polynomials in terms of the homogeneous cone rank. We note that it is not obvious that this is even possible in the first place. 
However, we have the following result.

\begin{proposition}\label{prop:exp_bound}
Let $\stdCone$ be a pointed homogeneous cone of rank $n$. 
Then, the degree $d$ of minimal polynomials of $\stdCone$ satisfy 
$d \leq 2^{n-1}$.
\end{proposition}
\begin{proof}
Without loss of generality, we may assume that $\stdCone$ has dimension $m$ and is contained in some $\Re^m$. In this way, $\stdCone$ has non-empty interior with respect to $\Re^m$. %We may also assume that $\stdCone$ is pointed \footnote{If $\stdCone$ is not pointed, letting $L\coloneqq \stdCone \cap - \stdCone$ be the lineality space of $\stdCone$, we have $\stdCone = (\stdCone \cap L) \oplus L^\perp$, where $\stdCone \cap L$ is a pointed homogeneous cone.}
In view of Proposition~\ref{prop:bound_hyper}, it is enough to exhibit a polynomial of degree $2^{n-1}$ that vanishes on the boundary of $\stdCone$.
Such a polynomial is described in \cite[Chapter~3, \S 3]{V63} and also in \cite[Section~1]{Ishi01}. Here we will follow the account in \cite{Ishi01}.

There exists $n$ polynomials $D_i : \Re^m \to \Re$ of degree 
$2^{i-1}$ such that the following property holds
\[
x \in \interior \stdCone \Leftrightarrow D_{i}(x) > 0, i \in \{1, \ldots, n\},
\]
see the discussion that goes from (1.11) in \cite{Ishi01} to \cite[Proposition~1.3]{Ishi01} or see the discussion around \cite[Chapter~3, \S 3, Proposition~2]{V63} for the analogous result in Vinberg's notation.
In particular, the $D_{i}$'s must be nonnegative on the boundary of $\stdCone$.
Intuitively, the $D_i$ are similar to the leading principal minors of symmetric matrices.

A key property is that the $D_i$'s satisfy a recurrence relation of the following form:
\begin{equation}\label{eq:dk}
D_k = D_1D_2 \cdots D_{k-1} p_k - \left(\sum _{i < k-1} D_{i+1} \cdots D_{k-1}q_{ki}^2\right) - q_{k,k-1}^2,
\end{equation}
see \cite[Proposition~1.4]{Ishi01}. 
Here, $p_k$, $q_{ki}$ and $q_{k,k-1}$ are polynomials whose specific formulae do not matter for our purposes.

Next, let $x$ be in the boundary of $\stdCone$.
We are now positioned to argue that $D_n(x)$ (which has degree $2^{n-1}$) must be zero. Since $D_i(x) \geq 0$ for all $i$ and $x$ is not in the interior of $\stdCone$, there exists some $k$ for which $D_{k}(x) = 0$ holds.
If $k = n$ we are done. Otherwise, in view of $\eqref{eq:dk}$, we have
\[
D_{k+1}(x) = -(q_{k+1,k}(x))^2.
\]
Since $D_{k+1}(x) \geq 0$, we have $D_{k+1}(x) = 0$. By induction, we conclude that if $D_{k}(x)$ vanishes, then all the $D_{j}(x)$ vanish for $j > k$. In particular $D_{n}(x)$ is zero.
\end{proof}
We note that G\"uler has a similar result in \cite[Section~8]{Gu97},
where he showed that homogeneous cones are hyperbolicity cones. 
However it is not immediately clear the degree of the polynomial obtained in \cite[Theorem~8.1]{Gu97}. {In more detail, G\"uler showed in \cite[Theorem~8.1]{Gu97}  that $p(x) = \prod_{i=1}^n \chi _i(x)^{d_i}$ is a hyperbolic polynomial for a homogeneous cone $\stdCone$, where $(d_1,\ldots, d_n) \coloneqq (1,1,2,4,\ldots,2^{n-2})$ and $\chi _i(x)$ are \emph{rational} functions satisfying $x \in \reInt \stdCone \Leftrightarrow \chi_{i}(x) > 0, \forall i$. From the discussion therein, it is not clear what the degree of $p$ is, since the degrees of the rational functions $\chi _i(x)$ are not described.}

Proposition~\ref{prop:exp_bound} is somewhat disappointing because it gives an exponential upper bound on the minimal  polynomial for homogeneous cones. 
We note that Ishi identified the irreducible components of $D_k$ and was able to express the interior of a  homogeneous cone using certain polynomials $\Delta_k$ of potentially smaller degree than the $D_k$, see \cite[Proposition~2.3]{Ishi01}.
%See also \cite{Na14, Na18} for related works on this topic. 
Related to that, Nakashima described techniques for computing the degrees of $\Delta_k$, see \cite[Equation~(1.5), Theorem~6.1]{Na14} and \cite[Lemma~1.1]{Na18}.
Nevertheless, we were not able to improve
the bound in Proposition~\ref{prop:exp_bound}. %produce better upper bounds in terms of the rank.

In contrast, in the case of cones arising from homogeneous chordal sparsity patterns, the situation is far more favourable and we have a quadratic bound in terms of the rank.

\begin{theorem}\label{theo:degree}
Let $G$ be a homogenous chordal graph on $n$ nodes. Then, a minimal polynomial for $\mathcal{S}^*(G)$ has degree no larger than $\lceil \frac{n+1}{2} \rceil \lfloor\frac{n+1}{2} \rfloor $.
\end{theorem}
\begin{proof}
We start by noting that the structure of a homogeneous chordal graph can be completely characterized by a rooted forest. 
A rooted forest is a directed acyclic graph (DAG),  whose connected components are trees having a marked node, which we call a root, and having all edges oriented away from the root. As is the case for every DAG, it induces a partial order in its nodes. In \cite[Section 2.3]{TV23} it is noted that for any homogeneous chordal graph there is a rooted forest $T$ with the same nodes as $G$ such that $G$ is the comparability graph of $T$, i.e., two nodes are connected if they are comparable in the partial order induced by $T$, in other words, if there is a directed path in $T$ between them. 

With this correspondence, maximal cliques of $G$ correspond to maximal paths away from a root in $T$. To see this, we note that the nodes form a clique in $G$ if and only if the partial order induced by $T$ restricted to those nodes becomes a total order, which in its turn means that we can order the nodes such that there is a path from each node to its successor, so they are all contained in an oriented path. 
With this correspondence, by Theorem~\ref{th:minpoly-chordal}, the degree of the minimal polynomial of $\mathcal{S}^*(G)$ is obtained by summing the numbers of nodes of every path from a root to a leaf, plus the number of isolated roots. We will now bound this quantity.

We observe that in a tree DAG, there is at most one directed path from any 
two nodes. In addition, all maximal paths must start on a root. This implies that the leaf of a maximal path cannot be a part of any other maximal path.
So if there are $k+1$ maximal paths, for each of them there are $k$ nodes it cannot contain, so the maximum length is $n-k$, and the maximum sum of lengths is $(k+1)(n-k)$.  We just have to maximize this number over all integer $k$ between $1$ and $n-1$ and this is attained at $k=\lfloor \frac{n-1}{2} \rfloor$, giving the degree $\lceil \frac{n+1}{2} \rceil \lfloor\frac{n+1}{2} \rfloor$ as stated. This is precisely the cone in Example~\ref{ex:max_degree}.
%First consider any rooted forest $T$ with $n$ nodes, and let $k$ be one of its leaves of maximum height, i.e., farthest away from a root. Given any other leave $l$, removing it from its parent and gluing it to the parent of $k$, does not decrease its height and does not reduce the number or heights of the other leaves, so it does not decrease the total sum of nodes in the paths to the leaves. If there are any isolated roots, we can also remove them and replaced them by leaves attached to the parent of $k$, as this will only increase the degree. We can therefore repeat these operations until we obtain a connected graph on which every leaf has the same height and is attached to the same parent.  
%The maximal sum of heights must therefore be attained in such a graph.
%
%Now, given a tree of that form, if the leaves have height $k$, that means that there are $k$ nodes in the \textquoteleft stalk\textquoteright \ of the tree, and $n-k$ leaves. The total sum is therefore $(n-k)(k+1)$. We just have to maximize this number over all integer $k$ between $1$ and $n-1$ and this is attained at $k=\lfloor \frac{n-1}{2} \rfloor$, giving the degree $\lceil \frac{n+1}{2} \rceil \lfloor\frac{n+1}{2} \rfloor$ as stated. This is precisely the cone in Example~\ref{ex:max_degree}.
\end{proof}

For general homogeneous cones it seems hard to give a polynomial bound on the degree of minimal polynomials in terms of the rank. However, as a consequence of the fact that the homogeneous cones can be realized as slices of positive semidefinite cones (i.e., homogeneous cones are spectrahedral), we have the following ``easy'' bound in terms of the dimension of the cone.

\begin{proposition}\label{prop:sdp_representation}
Let $\stdCone$ be a pointed homogeneous cone. The degree $d$ of minimal polynomials of $\stdCone$ satisfies $d \leq \dim \stdCone$. 
\end{proposition}
\begin{proof}
Without loss of generality, we may assume that $\stdCone$ is contained in some $\Re^k$ for which we have $k = \dim \stdCone$.
Chua proved in \cite[Corollary~4.3]{CH03} that if $\stdCone$ is a homogeneous cone in $\Re^k$, then for some $m \leq k$, there exists an injective linear map $M:\Re^k \to \cS^{m}$ with the property that \[
M(\reInt \stdCone) = \pdcone{m} \cap M(\Re^k)\]
holds, where $\pdcone{m}$ is the cone of $m\times m$ positive definite matrices and $\reInt \stdCone$ is the relative interior of $\stdCone$. See also \cite[Section~6]{TV23} which discusses related results contained in \cite{V63,Ro66,FB02} and other works. %and we remark that, as pointed out in \cite{TV23}, an earlier version of this result was contained in a paper by Rothaus \cite{Ro66}\footnote{Rothaus starts the paper by defining  what is a representation of a homogeneous cone: in our setting this corresponds to a linear map that takes the interior of the cone into the interior of a certain $\psdcone{n}$ and the map must be compatible with the action of some group that acts transitively on the interior.}.

In particular, we have that $M(\stdCone) = \psdcone{m} \cap M(\Re^k)$.
Put otherwise, $\stdCone$ is linearly isomorphic to an intersection of a hyperbolicity cone $(\psdcone{m})$ and a subspace. 
If we let $e \in \stdCone$ be such that $M(e) \in \pdcone{m}$, then 
the composition $q\coloneqq \det \circ M$ corresponds to a hyperbolic polynomial of degree $m$ along $e$ such that $\stdCone = \Lambda_{+}(q,e)$.
This shows that $d \leq m \leq k = \dim \stdCone$.
\end{proof}
Unfortunately, there is no straightforward relation between rank and dimension. For example, second-order cones of dimension $m+1$ always have rank $2$ (for $m > 0$). In particular, for fixed rank, one may construct homogeneous cones of arbitrary large dimensions. 

In view of the discussion so far, during the writing of the first version of this paper we considered the following (then open) question.
\begin{opQ}
Is there a polynomial bound on the degrees of minimal polynomials of homogeneous cones in terms of rank? Or, more optimistically, does a quadratic bound hold as in the case of homogeneous chordal cones?	
\end{opQ} 

Shortly after we completed the paper, we were informed by Nakashima that the answer is \emph{no} and, in particular, for any integer $r > 0$ it is possible to construct  a homogeneous cone of rank $r$ for which its minimal polynomials have 
degree $2^{r-1}$, see \cite{Na24} for more details.
This shows that the bound in Proposition~\ref{prop:exp_bound} cannot be improved.

%\paragraph{Homogeneous chordality}

%\todo[inline]{B: João, I've left the next paragraph as it is since it corresponds to something you've written previously.}
%Homogeneous chordal graphs are in (essentially) one to one correspondence with rooted trees: given a rooted tree $T$ of vertices labeled $\{1,...,n\}$, orient all edges away from the root, and consider the graph $G_T$ where the edges are all $\{i,j\}$ such that there is a directed path in $T$ from $i$ to $j$. Maximal cliques in $G_T$ correspond to maximal oriented paths in $T$, and we can move from one description to the other.

We now conclude this paper with a few remarks on interior point methods.
%\begin{remark}[On interior point methods]
	The \emph{barrier parameter} of a closed convex cone $\stdCone$ is the smallest $d$ for which there exists a self-concordant barrier function for $\stdCone$ with parameter $d$, see, for example, \cite{GT98,Gu97}. 
	The importance of the barrier parameter is that it can be used to bound the worst-case iteration complexity of interior-point methods.
	
	We recall here the optimal barrier parameter for a homogeneous cone is equal to its rank, see \cite[Theorem~4.1]{GT98}. In contrast, if $\stdCone = \Lambda_+(p,e)$ is a hyperbolicity cone and $p$ is a hyperbolic polynomial of degree $d$, then $-\log p$ is a self-concordant barrier for $\stdCone$ of parameter
	$d$, see \cite[Section~4]{Gu97}. 
	The function $x \mapsto -\log p(x)$  is called a \emph{hyperbolic barrier function} for $\stdCone$.
	
	Suppose that $G$ is a homogeneous chordal cone with $n$ nodes. 
	Then, the homogeneous cone rank of $\mathcal{S}^*(G)$ is $n$ and, therefore, 
	the barrier parameter of $\mathcal{S}^*(G)$ is $n$ as well.
	However, as a consequence of Theorem~\ref{th:minpoly-chordal}, no hyperbolic barrier function has parameter $n$ unless $G$ is a disjoint union of cliques. 
	Put otherwise, hyperbolic barrier functions for homogeneous cones are typically not optimal, even if the polynomial considered is minimal. %Theorem~\ref{theo:degree}, on the other hand, indicates that the minimal degree is no larger than $O(n^2)$. 
	
	For $\mathcal{S}^*(G)$, in view of Theorem~\ref{theo:degree}, the difference in order between the optimal barrier parameter and the one afforded by the best possible hyperbolic barrier function is polynomial. However, the discussion in \cite{Na24} indicates that, in general, the difference in order may be exponential.% for a general homogeneous cone the difference between rank and minimal degree can be quite large. 
	%For example, the barrier parameter of the Vinberg cone is $3$, but its minimal polynomials have degree 4.
	
	The existence of a gap between both quantities suggests that it may be more advantageous to deal with a homogeneous cone constraint on its own terms (e.g., as in \cite{Ch09,TV23}) instead of seeing it as a hyperbolicity cone or as a slice of a positive semidefinite cone.
	
	Some numerical evidence towards this was given in \cite[Section~14.4]{KT23}, where a problem of minimizing the nuclear norm is solved through two different formulations. 
	Denoting by $\Re^{m\times n}$ the space of $m\times n$ matrices and by $U^T$ the transpose of $U \in \Re^{m\times n}$, the first formulation uses a self-concordant barrier for the set
	$\{(Z,U) \in \cS^m \times \Re^{m\times n} \mid Z-UU^T \in \psdcone{m}\}$ (see Section~7 therein), which is a slice of the homogeneous cone of rank $m+1$ given in \cite[Example~5]{GT98}. 
	The second uses a formulation via semidefinite programming.
	The results in \cite[Fig.~2]{KT23} seem to indicate that the former is significantly more efficient than the latter when $n$ is much larger than $m$.
	 
	%epigraph of the nuclear norm (which is a slice of a homogenous cone \footnote{Denoting by $\norm{A}_*$ the nuclear of $m\times n$ matrix, we have $\norm{A}_* \leq t \Leftrightarrow $ } )
	
%\end{remark}
{\small
\section*{Acknowledgements}
We would like to thank Hideto Nakashima for helpful comments.
We also thank the referee for their  comments, which helped to improve the paper.
The support of the Institute of Statistical Mathematics in the form of a visiting professorship for the first author is gratefully acknowledged.
}

%\todo[inline]{Comment on the gap between homogenous cone rank and minimal degree polynomials. Try to understand how bad can Guler's construction be in terms of degree.}
%
%\section{Some consequences}
%\todo[inline]{Mention that the Vinberg cone corresponds to the $G_4$ case in Example~\ref{ex:max_degree}}
%\begin{example}
%Consider the convex cone
%$$
%K = \left\{(a,b,c,x,y) \in \RR^5 ~:~ 
%\left(\begin{array}{cc}a & x \\ x & c\end{array}\right) \succeq O,~
%\left(\begin{array}{cc}c & y \\ y & b\end{array}\right) \succeq O
%\right\}
%$$
%\begin{itemize}
%\item $K$ is a homogeneous cone of rank $3$. \todo{reference}
%\item $K=\Lambda(p_1,e)\cap \Lambda(p_2,e) = \Lambda(p_1p_2,e)$ for $p_1(a,b,c,x,y)=ac-x^2$, $p_2(a,b,c,x,y)=bc-y^2$ and $e=(1,1,1,0,0)$.
%\item $p_1$ and $p_2$ are irreducible on $\RR[a,b,c,x,y]$.
%\item The degree 4 polynomial $p_1p_2$ is a minimal degree polynomial for $K=\Lambda(p_1p_2,e)$ by Proposition~\ref{prop:bound_hyper}.
%\end{itemize}
%\end{example}
%
%\todo[inline]{What are the bounds on the the degree of minimal polynomials for homogeneous cones? Maybe discuss in the conclusion while mentioning the known results.}

%\todo[inline]{Are the degrees of the $\Delta$'s in Ishi's paper bounded by the rank? Guler's bound seems to give exponential bounds. it may be interesting to comment on that.}

\bibliographystyle{alpha}
\bibliography{selfdual}{}

\end{document}